\documentclass[10.5pt, a4paper]{amsart}
\usepackage{amssymb, amsmath, amscd, bm, amsthm, pdfpages, setspace, hyperref, mathtools, subcaption, graphicx, cite}

\def\@Rref#1{\hbox{\rm \ref{#1}}}
\def\Rref#1{\@Rref{#1}}
\theoremstyle{plain}
\newtheorem{theorem}{Theorem}[section]
\newtheorem{proposition}[theorem]{Proposition}

\newtheorem{lemma}[theorem]{Lemma}

\theoremstyle{definition}
\newtheorem{definition}{Definition}[section]
\newtheorem{example}[definition]{Example}

\newcommand{\im}{\mathop{\rm Im}\nolimits}
\newcommand{\re}{\mathop{\rm Re}\nolimits}

\begin{document}

\title[Decay Rate of $\exp(A^{-1}t)A^{-1}$ and Crank-Nicolson Scheme]{Decay Rate of $\bm{\exp(A^{-1}t)A^{-1}}$ on a Hilbert Space
	and the Crank-Nicolson Scheme with Smooth Initial Data}

\thispagestyle{plain}

\author{Masashi Wakaiki}
\address{Graduate School of System Informatics, Kobe University, Nada, Kobe, Hyogo 657-8501, Japan}
 \email{wakaiki@ruby.kobe-u.ac.jp}
 \thanks{This work was supported by JSPS KAKENHI Grant Number JP20K14362.}

\begin{abstract}
This paper is concerned with the decay rate of $e^{A^{-1}t}A^{-1}$ 
for the generator $A$ of an exponentially stable $C_0$-semigroup 
on a Hilbert space.
To estimate the decay rate of $e^{A^{-1}t}A^{-1}$,
we apply a bounded functional calculus.
Using this estimate and Lyapunov equations, we also study
the quantified asymptotic behavior of the Crank-Nicolson
scheme with smooth initial data.
A similar argument is applied to a
polynomially stable $C_0$-semigroup whose generator is normal.
\end{abstract}



\maketitle

\section{Introduction}
Let $A$ be the generator of a bounded $C_0$-semigroup $(e^{At})_{t\geq 0}$ 
on
a Hilbert space. Suppose that the inverse
$A^{-1}$ of $A$ exists and also generates a $C_0$-semigroup
$(e^{A^{-1}t})_{t\geq 0}$.
The problem we study first is to estimate the decay rate of 
$\|e^{A^{-1}t}A^{-1}\|$ under additional assumptions on
stability of  $(e^{At})_{t\geq 0}$.
This is a variant of the so-called inverse generator problem raised by
deLaubenfels \cite{deLaubenfels1988}: ``Let $A$ be the generator of
a bounded $C_0$-semigroup
on
a Banach space. Assume that there exists an inverse $A^{-1}$
as a closed, densely-defined operator. Does $A^{-1}$
also generate a bounded $C_0$-semigroup?''
Positive answers to the inverse generator problem have been given
for bounded holomorphic $C_0$-semigroups on Banach spaces  \cite{deLaubenfels1988}
and for contraction $C_0$-semigroups on Hilbert spaces \cite{Zwart2007}.
Negative examples can be found in  \cite{Komatsu1966}
with respect to the generation property of $A^{-1}$ and in
\cite{Zwart2007} with respect to
the boundedness property of $(e^{A^{-1}t})_{t\geq 0}$;
see also \cite{Gomilko2007MS,Fackler2016} for other counterexamples.
For bounded $C_0$-semigroups on Hilbert spaces, 
the answer to the inverse generator problem  is still unknown, but it has been shown in 
\cite{Zwart2007SF} that if $A$ generates an exponentially stable $C_0$-semigroup on Hilbert spaces, then
$\|e^{A^{-1}t}\| = O(\log t)$ as $t \to \infty$, that is,
there exist $M > 0$ and $t_0 > 1$ such that $\|e^{A^{-1}t}\| \leq  M \log t$
for all $t \geq  t_0$.
This result has been extended to 
bounded $C_0$-semigroups  with boundedly invertible generators in \cite{Batty2021}.
We refer the reader to the survey article \cite{Gomilko2017} for more detailed
discussion of the inverse generator problem.

It is known that if $A$ generates an exponentially stable $C_0$-semigroup
on a Banach space, then
\begin{equation}
\label{eq:Banach}
\big\|e^{A^{-1}t} A^{-k}\big\| = O \left(
\frac{1}{t^{k/2 - 1/4}}
\right)\qquad  (t \to \infty)
\end{equation}
for all $k \in \mathbb{N}$.
This norm-estimate has been obtained in \cite{Zwart2007} for $k=1$
and in \cite{deLaubenfels2009} for $k \geq 2$.
In this paper, we obtain an analogous estimate for the Hilbert case:
if $A$ generates an exponentially stable $C_0$-semigroup
on a Hilbert 
space, then
\begin{equation}
\label{eq:Hilbert}
\big\|e^{A^{-1}t} A^{-k}\big\| = O \left(
\frac{1}{t^{k/2}}
\right)\qquad  (t \to \infty)
\end{equation}
for all $k \in \mathbb{N}$.
To obtain the estimate \eqref{eq:Banach} of the Banach case,
an integral representation of $(e^{A^{-1}t})_{t \geq 0}$
(see 
\cite[Lemma~3.2]{Zwart2007}
and 
\cite[Theorem~1]{Gomilko2007MS}) has been applied.
In contrast, to obtain the estimate \eqref{eq:Hilbert} of the Hilbert case,
we employ a functional-calculus approach based on
the $\mathcal{B}$-calculus introduced in \cite{Batty2021}.
Another class of $C_0$-semigroups we study are polynomially
stable $C_0$-semigroups with parameter $\beta>0$ on Hilbert spaces, where 
$(e^{At})_{t\geq 0}$ is called a polynomially stable
$C_0$-semigroup with parameter $\beta >0$ if
$(e^{At})_{t\geq 0}$ is bounded and if
$\|e^{At}(I-A)^{-1}\| = O(t^{-1/\beta})$ as $t \to \infty$.
By a functional-calculus approach  as in
the case for exponentially stable $C_0$-semigroups,
we show that if, in addition, $A$ is a normal operator, then
\begin{equation}
\label{eq:Hilbert_normal}
\big\|e^{A^{-1}t} A^{-k}\big\| = O \left(
\frac{1}{t^{k/(2+\beta)}}
\right)\qquad  (t \to \infty)
\end{equation}
for all $k \in \mathbb{N}$.
We also present simple examples for which
the norm-estimates \eqref{eq:Hilbert} 
and \eqref{eq:Hilbert_normal} cannot be improved.
Moreover, we prove that  the (not necessarily normal) generator $A$
of a polynomially stable $C_0$-semigroup with any parameter 
on a Hilbert space
satisfies $\sup_{t \geq 0} \|e^{A^{-1}t}A^{-1}\| < \infty$,
inspired by the Lyapunov-equation technique developed in \cite{Zwart2007SF} .

Next, we study the
Crank-Nicolson discretization scheme with variable stepsizes.
Let $A$ be the generator of a bounded $C_0$-semigroup on a Hilbert space,
and let $\tau_n$, $n \in \mathbb{N}_0\coloneqq \{ 0,1,2,\cdots\}$, satisfy 
$\tau_{\min} \leq \tau_n
\leq \tau_{\max}$ for some $\tau_{\max} \geq \tau_{\min} >0$.
Consider the time-varying difference equation 
\begin{equation}
\label{eq:diff_eq_intro}
x_{n+1} = A_d(\tau_n)x_n,\quad n \in \mathbb{N}_0,
\end{equation}
where $A_d(\tau)$ is the 
Cayley transform of $A$ with parameter $\tau>0$, i.e.,
\[
A_d(\tau) \coloneqq \left(
I + \frac{\tau}{2}A
\right) \left(
I - \frac{\tau}{2}A
\right)^{-1}. 
\]
It is our aim to estimate the decay rate of the solution $(x_n)_{n \in \mathbb{N}}$ 
with smooth initial data $x_0 \in D(A)$.

For a  constant stepsize $\tau >0$, upper bounds on
the growth rate of $(A_d(\tau)^n)_{n \in \mathbb{N}_0}$
have been obtained
for bounded $C_0$-semigroups 
and exponentially stable $C_0$-semigroups 
on Banach spaces; see \cite{Brenner1979, Gomilko2011}.
The relation between 
$(e^{A^{-1}t})_{t \geq 0}$ and
$(A_d(\tau)^n)_{n \in \mathbb{N}_0}$
in terms of growth rates
has been investigated in \cite{Gomilko2011}.
As in the inverse generator problem,
$(A_d(\tau)^n)_{n \in \mathbb{N}_0}$ is bounded whenever
$C_0$-semigroups are bounded holomorphic on Banach spaces 
\cite{Crouzeix1993}. Furthermore, it is well known that for contraction 
$C_0$-semigroups on  
Hilbert spaces, $A_d(\tau)$ is
a contraction; see \cite{Phillips1959}.
It is still open whether $(A_d(\tau)^n)_{n \in \mathbb{N}_0}$
is bounded for every bounded $C_0$-semigroup on a Hilbert space.
However, it has been shown independently in \cite{Azizov2004,Gomilko2004,
	Guo2006} that
if $A$ and $A^{-1}$ both generate bounded $C_0$-semigroups on 
a Hilbert space, then $(A_d(\tau)^n)_{n \in \mathbb{N}_0}$ is bounded.
If in addition $(e^{At})_{t\geq 0}$ is strongly stable, then
$(A_d(\tau)^n)_{n \in \mathbb{N}_0}$ is  strongly stable \cite{Guo2006}.
Without the assumption on  $A^{-1}$, the norm-estimate
$\|A_d(\tau)^n\| = O(\log n)$ as $n \to \infty$ 
given in \cite{Gomilko2004}
remains the best so far in
the Hilbert case. 
For polynomially stable $C_0$-semigroups with parameter $\beta>0$
on Hilbert spaces, 
the following estimates have been obtained in \cite{Wakaiki2021JEE}:
\begin{itemize}
	\item
	$
	\|A_d(\tau)^nA^{-1}\| = 
	O\bigg(\left( \dfrac{\log n}{n}\right)^{1/(2+\beta)}\bigg)
	$ 
	if $(A_d(\tau)^n)_{n \in \mathbb{N}_0}$ is bounded. \vspace{3pt}
	\item
	$
	\|A_d(\tau)^nA^{-1}\| = 
	O\left(
	\dfrac{1}{n^{1/(2+\beta)}}
	\right)
	$
	if $A$ is normal.
\end{itemize}
For more information on
asymptotics of $(A_d(\tau)^n)_{n \in \mathbb{N}_0}$,
we refer to the survey~\cite{Gomilko2017}.

Some of the above results have been extended to 
the case of variable stepsizes.
For bounded $C_0$-semigroups on Banach spaces,
the Crank-Nicolson scheme with variable stepsizes has the same growth bound
as that with constant stepsizes; see  \cite{Bakarev2001}.
It has been proved that 
the solution $(x_n)_{n \in \mathbb{N}_0}$ of the time-varying 
difference equation \eqref{eq:diff_eq_intro}   is bounded 
if $A$ generates a bounded holomorphic $C_0$-semigroups on a Banach space 
\cite{Palencia1993, Piskarev2007, Casteren2011}
or if
$A$ and $A^{-1}$ generate bounded $C_0$-semigroups on a Hilbert space \cite{Piskarev2007}.
The stability property, i.e.,
$
\lim_{n \to \infty} x_n = 0
$
for all initial data $x_0$,
has been also obtained under some additional assumptions in \cite{Piskarev2007}.

We show that if $A$ and $A^{-1}$ are the generators of
an exponentially stable $C_0$-semigroup
and a bounded $C_0$-semigroup on a Hilbert space, respectively, then
the solution $(x_n)_{n \in \mathbb{N}_0}$ of the time-varying
difference equation \eqref{eq:diff_eq_intro} with  $x_0 \in D(A)$
satisfies
\begin{equation}
\label{eq:CN_exp}
\|x_n\| = O\left(
\sqrt{\frac{\log n}{n}}
\right)
\qquad (n \to \infty).
\end{equation}
Moreover, if $A$ is a 
normal operator generating
a polynomially stable $C_0$-semigroup with parameter
$\beta >0$ on a Hilbert space, then
\begin{equation}
\label{eq:CN_poly}
\|x_n\| = O\left(
\frac{1}{n^{1/(2+\beta)}}
\right)
\qquad (n \to \infty)
\end{equation}
for all $x_0 \in D(A)$.
These estimates \eqref{eq:CN_exp} and 
\eqref{eq:CN_poly} follow from the combination of
the norm-estimates \eqref{eq:Hilbert} and \eqref{eq:Hilbert_normal} 
of $e^{A^{-1}t}A^{-1}$ and
the Lyapunov-equation technique developed in \cite{Piskarev2007}.
In the constant case $\tau_n \equiv \tau >0$, we give a simple example of an exponentially stable $C_0$-semigroup showing that the decay rate $\|A_d(\tau)^n A^{-1}\| = O(1/\sqrt{n})$
cannot in general be improved.
The estimate \eqref{eq:CN_exp} includes the logarithmic factor $\sqrt{\log n}$, but
its necessity remains open.
On the other hand, the decay rate given in \eqref{eq:CN_poly}
is the same as that obtained for constant stepsizes in
\cite{Wakaiki2021JEE}, and one cannot in general replace 
$n^{-1/(2+\beta)}$ in \eqref{eq:CN_poly} by
functions with better decay rates.

This paper is organized as follows. In Section~\ref{sec:inverse_generator},
we give long-time norm-estimates for $e^{A^{-1}t}A^{-1}$.
We consider  first exponentially stable $C_0$-semigroups and then
polynomially stable $C_0$-semigroups. In Section~\ref{sec:CNS},
the decay rate estimates of $\|e^{A^{-1}t}A^{-1}\|$ are utilized 
in order to examine the quantified asymptotic behavior of the Crank-Nicolson
scheme with smooth initial data.

\paragraph{Notation}
Let
$\mathbb{C}_+ \coloneqq 
\{
\lambda \in \mathbb{C}: \re \lambda >0
\}$, 
$\overline{\mathbb{C}}_+ \coloneqq 
\{
\lambda \in \mathbb{C}: \re \lambda \geq 0
\}$, 
$\mathbb{C}_- \coloneqq 
\{
\lambda \in \mathbb{C}: \re \lambda <0
\}$, 
and 
$i \mathbb{R} \coloneqq \{ i \eta : \eta \in \mathbb{R}  \}$.
Let $\mathbb{N}_0$ and $\mathbb{R}_+$ be 
the set of nonnegative integers and the set of nonnegative real numbers,
respectively.
For real-valued functions $f,g$ on $J \subset \mathbb{R}$, we write
\[
f(t) = O\big(g(t)\big)\qquad (t \to \infty)
\]
if there exist constants $M >0$ and $t_0 \in J$ such that 
$f(t) \leq M g(t)$ for all $t \in J$ satisfying $t \geq t_0$.
The gamma function is denoted by $\Gamma$, i.e.,
\[
\Gamma(\alpha ) = \int^{\infty}_0 t^{\alpha -1} e^{-t}dt,\quad \alpha > 0.
\]
Let $H$ be a Hilbert space. The inner product of $H$ is denoted 
by $\langle \cdot, \cdot \rangle$.
The space of bounded linear operators on $H$ is denoted by $\mathcal{L}(H)$.
For a linear operator $A$ on $H$, let $D(A)$ and $\sigma(A)$
denote the domain and the spectrum of $A$, respectively.
For a densely-defined linear operator $A$ on $H$,
we denote by $A^*$ the Hilbert space adjoint of $A$.
If $A$ is the generator of a bounded $C_0$-semigroup on a Hilbert space
and if $A$ is injective, 
the fractional power $(-A)^{\alpha}$ of $-A$ is defined by
the sectorial functional calculus 
for $\alpha\in \mathbb{R}$ as in \cite[Chapter~3]{Haase2006}.
Let $\ell^2$ be the space of complex-valued square-summable sequences
$x = (x_n)_{n \in \mathbb{N}}$
endowed with the norm $\|x\| \coloneqq
\sqrt{\sum_{n=1}^{\infty} |x_n|^2}$. 

\section{Norm-Estimates of Semigroups Generated by Inverse Generators}
\label{sec:inverse_generator}
The aim of this section is to derive a norm-estimate for
$e^{A^{-1}t} (-A)^{-\alpha}$ with $\alpha >0$.
First we assume that 
$A$ is the generator of an exponentially stable $C_0$-semigroup 
on a Hilbert space.
Next we focus on the generator $A$ of a polynomially stable $C_0$-semigroup on 
a Hilbert space.

\subsection{Inverses of generators of exponentially stable semigroups}
Let $A$ be the generator of 
an exponentially stable $C_0$-semigroup on a Hilbert space.
Then $A$ is invertible, and $A^{-1}$ generates a $C_0$-semigroup 
$(e^{A^{-1}t})_{t \geq 0}$.
To obtain a norm-estimate for $e^{A^{-1}t} (-A)^{-\alpha}$,
we employ the $\mathcal{B}$-calculus introduced in  \cite{Batty2021}.

\subsubsection{Basic facts on $\mathcal{B}$-calculus}
Let $\mathcal{B}$ be the space of all holomorphic functions
$f$ on $\mathbb{C}_+$ such that 
\[
\|f\|_{\mathcal{B}_0} \coloneqq 
\int^{\infty}_0 \sup_{\eta \in \mathbb{R}} |f'(\xi+i\eta)|d\xi < \infty.
\]
We recall elementary properties of functions $f$ in $\mathcal{B}$.
The proof can be found in \cite[Proposition~2.2]{Batty2021}.
\begin{proposition}
	For $f \in \mathcal{B}$, the following statements hold:
	\begin{enumerate}
		\renewcommand{\labelenumi}{(\roman{enumi})}
		\item $f(\infty) \coloneqq \lim_{\re z \to \infty} f(z)$ exists in $\mathbb{C}$.
		\item $f$ is bounded on $\mathbb{C}_+$, and $\|f\|_{\infty} \coloneqq
		\sup_{z \in \mathbb{C}_+} |f(z)| \leq |f(\infty)| + \|f\|_{\mathcal{B}_0} $.
		\item
		$f(i\eta) \coloneqq \lim_{\xi \to 0 +} f(\xi+i\eta)$ exists, uniformly for $\eta \in \mathbb{R}$.
	\end{enumerate}
\end{proposition}

A norm on $\mathcal{B}$ is defined by
\[
\|f\|_{\mathcal{B}} \coloneqq \|f\|_{\infty} + \|f\|_{\mathcal{B}_0},
\]
which is equivalent to each norm of the form
\[
|f(z_0)| +  \|f\|_{\mathcal{B}_0}\qquad (z_0 \in \overline{\mathbb{C}}_+ \cup \{
\infty
\}).
\]
The space $\mathcal{B}$ equipped with the norm $\|\cdot\|_{\mathcal{B}}$ 
is a Banach algebra by
\cite[Proposition~2.3]{Batty2021}.

Let $\textrm{M}(\mathbb{R}_+)$ be the Banach algebra of all bounded Borel measures on
$\mathbb{R}_+$ under convolution, endowed with the norm
$\|\mu\|_{\textrm{M}(\mathbb{R}_+)} \coloneqq
|\mu|(\mathbb{R}_+)$,
where $|\mu|$ is the total variation of $\mu$. For $\mu \in \textrm{M}(\mathbb{R}_+)$, 
the Laplace transform of $\mu$ is the function
\[
\mathcal{L}\mu\colon \overline{\mathbb{C}}_+ \to \mathbb{C},\quad 
(\mathcal{L}\mu)(z) = \int_{\mathbb{R}_+} e^{-zt} \mu(dt).
\] 
We define
\[
\mathcal{LM} \coloneqq \{ 
\mathcal{L}\mu : \mu \in \textrm{M}(\mathbb{R}_+)
\}.
\]
Then the space  $\mathcal{LM}$ endowed with 
the norm 
$\|\mathcal{L}\mu\|_{\textrm{HP}} \coloneqq \|\mu\|_{\textrm{M}(\mathbb{R}_+)}$
becomes a Banach algebra. As shown on 
\cite[p.~42]{Batty2021},
$\mathcal{LM}$ 
is a subspace of $\mathcal{B}$ with continuous inclusion.

Let $-B$ be the generator of a bounded $C_0$-semigroup $(e^{-Bt})_{t \geq 0}$ on a 
Hilbert space
$H$. 
Define $K \coloneqq \sup_{t \geq 0} \|e^{-Bt}\|$.
Using Plancherel's theorem and the Cauchy-Scwartz inequality,
we obtain
\[
\int_{\mathbb{R}} \left|
\langle
(\xi + i \eta +B)^{-2}x,y
\rangle
\right| d\eta \leq \frac{\pi K^2}{\xi } \|x\|\, \|y\|
\]
for all $x,y \in H$ and $\xi >0$; see \cite[Example~4.1]{Batty2021}.
From this observation, we 
define 
\begin{align}
\langle
f(B)x,y 
\rangle \coloneqq
f(\infty) \langle x,y \rangle -
\frac{2}{\pi}
\int^{\infty}_0 \xi 
\int_{\mathbb{R}} \langle
(\xi - i\eta + B)^{-2} x, y
\rangle f'(\xi+i\eta) d\eta d\xi
\label{eq:functional_def}
\end{align}
for $f \in \mathcal{B}$ and $x,y \in H$. 
This definition yields a bounded
functional calculus; see \cite[Theorem~4.4]{Batty2021} for the proof.
\begin{theorem}
	\label{thm:B_calculus}
	Let $-B$ be the generator of
	a bounded $C_0$-semigroup $(e^{-Bt})_{t \geq 0}$ on a 
	Hilbert space $H$, and let $f(B)$ be defined as in \eqref{eq:functional_def}.
	Then the following statements hold:
	\begin{enumerate}
		\renewcommand{\labelenumi}{(\roman{enumi})}
		\item The map $\Phi_B\colon f \mapsto f(B)$ is a bounded algebra homomorphism from $\mathcal{B}$ into $\mathcal{L}(H)$.
		\item If $f = \mathcal{L}\mu \in \mathcal{LM}$, then $f(B)$ coincides with the operator
		\[
		x \mapsto \int_{\mathbb{R}_+} e^{-Bt}x \mu(dt)
		\]
		as defined in the Hille-Phillips calculus.
		\item 
		If $K \coloneqq \sup_{t \geq 0} \|e^{-Bt}\|$, then
		\[
		\|f(B)\| \leq |f(\infty)| + 2K^2\|f\|_{\mathcal{B}_0} \leq 2K^2\|f\|_{\mathcal{B}}
		\]
		for all $f \in \mathcal{B}$.
	\end{enumerate}
\end{theorem}
The  map
\[
\Phi_B : \mathcal{B} \to \mathcal{L}(H),\quad f \mapsto f(B)
\]
is called {\em the $\mathcal{B}$-calculus} for $B$.

\subsubsection{Norm-estimate by $\mathcal{B}$-calculus}
By Theorem~\ref{thm:B_calculus} (iii),
the norm-estimate for a function 
$f \in \mathcal{B}$ implies the 
norm-estimate for the corresponding operator $f(B)$.
In \cite[Lemma~3.4]{Batty2021},
the estimate 
\begin{align}
\label{eq:alp_zero}
\|h_t\|_{\mathcal{B}} 
= O\left(\log t\right)\qquad (t \to \infty),\qquad
\text{where $h_t(z) \coloneqq e^{-t/(z+1)}$,}
\end{align}
has been derived for the norm-estimate of $e^{A^{-1}t}$. 
To estimate $\|e^{A^{-1}t} (-A)^{-\alpha}\|$ for $\alpha >0$, 
the following result is used.
\begin{lemma}
	\label{lem:ftk_norm}
	For $t,\alpha > 0$, define
	\[
	f_{t,\alpha}(z) \coloneqq \frac{e^{-t/(z+1)}}{(z+1)^\alpha},\quad z \in \mathbb{C}_+.
	\]
	Then $f_{t,\alpha} \in \mathcal{B}$ for all $t,\alpha >0$.
	Moreover, 
	\begin{align}
	\label{eq:f_t_B0_norm}
	\|f_{t,\alpha}\|_{\mathcal{B}_0} = O\left(\frac{1}{t^{\alpha/2}}\right)\qquad (t \to \infty)
	\end{align}
	for each $\alpha >0$.
\end{lemma}
\begin{proof}
	For $t >0$ and $\alpha > 1$,  define
	\[
	F_\alpha (t) \coloneqq
	\int^{\infty}_0 \sup_{\eta \in \mathbb{R}}
	\left|\frac{e^{-t/(\xi+i\eta+1)}}{(\xi+i\eta+1)^\alpha}\right| d\xi.
	\]
	Since
	\[
	f_{t,\alpha}'(z) = \frac{te^{-t/(z+1)} }{(z+1)^{\alpha+2}} - \frac{\alpha e^{-t/(z+1)} }{(z+1)^{\alpha+1}}
	\]
	for all $z \in \mathbb{C}_+$,
	we obtain
	\begin{equation}
	\label{eq:ftk_bound}
	\int^{\infty}_0 \sup_{\eta \in \mathbb{R}} |f_{t,\alpha}'(\xi+i\eta)|d\xi 
	\leq tF_{\alpha+2}(t) + \alpha F_{\alpha+1}(t)
	\end{equation}
	for all $t,\alpha > 0$.
	
	Let $\alpha>1$.
	Define 
	\[
	g_{t,s,\alpha}(r) \coloneqq \frac{e^{-ts/(s^2+r)}}{(s^2+r)^{\alpha/2}}
	\]
	for $t>0$, $s >1$, and $r \geq 0$.
	Then
	\[
	F_{\alpha}(t)
	=
	\int^{\infty}_1 \sup_{r \geq 0} g_{t,s,\alpha}(r)ds 
	\]
	for all $t > 0$.
	Put $c \coloneqq  \left(\alpha/(2e)\right)^{\alpha/2}$.
	We have that for each $s > 1$,
	\[
	\sup_{r \geq 0} g_{t,s,\alpha}(r) = 
	\begin{cases}
	g_{t,s,\alpha}(0) = \dfrac{e^{-t/s}}{s^\alpha}, & 0 < t \leq \dfrac{\alpha s}{2}, \vspace{6pt}\\
	g_{t,s,\alpha}\left(\dfrac{2ts}{\alpha}-s^2 \right) = \dfrac{c}{(ts)^{\alpha/2}}, &  t \geq \dfrac{\alpha s}{2} > \dfrac{\alpha}{2}.
	\end{cases}
	\]
	If $0 < t \leq \alpha/2$, then
	\begin{align*}
	\int^{\infty}_1 \sup_{r \geq 0} g_{t,s,\alpha}(r)ds &=
	\int^{\infty}_1 \frac{e^{-t/s}}{s^\alpha}ds 
	= 
	\frac{1}{t^{\alpha-1}} \int^t_0 u^{\alpha-2}e^{-u} du 
	\leq 
	\frac{\Gamma(\alpha-1)}{t^{\alpha-1}}.
	\end{align*}
	If $t > \alpha /2$, then
	\begin{align*}
	\int^{\infty}_1 \sup_{r \geq 0} g_{t,s,\alpha}(r)ds &=
	\frac{c}{t^{\alpha /2}} \int^{2t/\alpha}_{1} \frac{1}{s^{\alpha/2}}ds + 
	\int^{\infty}_{2t/\alpha } \frac{e^{-t/s}}{s^\alpha } ds.
	\end{align*}
	We write the first term on the right-hand side as
	\begin{align*}
	\frac{c}{t^{\alpha /2}} \int^{2t/\alpha }_{1} \frac{1}{s^{\alpha /2}}ds 
	&=
	\begin{cases}
	\dfrac{c \log t}{t}, & \alpha = 2, \vspace{6pt}\\
	\dfrac{c}{\alpha/2-1} \left(
	\dfrac{1}{t^{\alpha/2}} - \dfrac{(\alpha/2)^{\alpha/2-1}}{t^{\alpha-1}} 
	\right), & \alpha \not=2.
	\end{cases}
	\end{align*}
	The second term satisfies
	\begin{align*}
	\int^{\infty}_{2t/\alpha} \frac{e^{-t/s}}{s^\alpha} ds =
	\frac{1}{t^{\alpha-1}} \int^{\alpha/2}_0 u^{\alpha-2}e^{-u} du 
	\leq 
	\frac{\Gamma(\alpha-1)}{t^{\alpha-1}}.
	\end{align*}
	Therefore, 
	$(0\leq)\, F_\alpha(t) < \infty$  for each $t >0$, and 
	\begin{equation}
	\label{eq:Fk_order}
	F_\alpha(t) =
	\begin{cases}
	O\left(\dfrac{1}{t^{\alpha-1}}\right), & 1<\alpha < 2, \vspace{6pt}\\
	O\left(\dfrac{\log t}{t}\right), & \alpha = 2, \vspace{6pt}\\
	O\left(\dfrac{1}{t^{\alpha/2}}\right), & \alpha >2.
	\end{cases}
	\end{equation}
	
	From the estimate \eqref{eq:ftk_bound}, we obtain
	$f_{t,\alpha } \in \mathcal{B}$
	for all $t,\alpha >0$.
	Since \eqref{eq:Fk_order} yields
	\begin{align*}
	tF_{\alpha+2}(t) + \alpha F_{\alpha+1}(t) =
	O\left(\frac{1}{t^{\alpha/2}}\right)\qquad (t \to \infty)
	\end{align*}
	for all $\alpha >0$,
	we obtain the desired estimate \eqref{eq:f_t_B0_norm}.
\end{proof}

In \cite[Corollary~5.7]{Batty2021},
the $\mathcal{B}$-calculus and 
the norm-estimate \eqref{eq:alp_zero} of $h_t(z) = e^{-t/(z+1)}$ have been used to obtain the norm-estimate
for $e^{A^{-1}t}$. 
Analogously,
the following theorem gives the norm-estimate for $e^{A^{-1}t}(-A)^{-\alpha}$
with $\alpha > 0$. 
\begin{theorem}
	\label{thm:inv_gen_bound}
	Let $A$ be the generator of 
	an exponentially stable $C_0$-semigroup $(e^{At})_{t \geq 0}$
	on a Hilbert space $H$. 
	Then 
	\begin{align}
	\label{eq:Inv_decay}
	\big\|e^{A^{-1}t} (-A)^{-\alpha}\big\| = 
	O\left(
	\frac{1}{t^{\alpha/2}}
	\right)
	\qquad  (t \to \infty)
	\end{align}
	for all $\alpha > 0$.
\end{theorem}
\begin{proof}
	There exist constants $K,\omega >0$ such that 
	\[
	\big\|e^{tA}\big\| \leq Ke^{-\omega t}
	\]
	for all $t \geq 0$.
	We may assume that $\omega = 1$ by replacing $A$ by $\omega^{-1}A$
	and $t$ by $\omega t$. 
	Note that $B\coloneqq -A-I$ generates a
	bounded $C_0$-semigroup $(e^{-tB})_{t \geq 0}$ with 
	$\sup_{t \geq 0} \|e^{-tB}\| \leq K$. 
	
	Take $t,\alpha > 0$ arbitrarily, and
	consider the functions 
	\[
	h_t(z) \coloneqq e^{-t/(z+1)},\quad 
	r_\alpha(z) \coloneqq \frac{1}{(z+1)^{\alpha}},\quad z \in \mathbb{C}_+.
	\]
	Then
	$h_t$ belongs to $\mathcal{LM} \subset \mathcal{B}$; 
	see \cite[Example~2.12]{Batty2021}.
	From the definition of the gamma function,
	we have that for all $z >0$,
	\[
	\Gamma(\alpha) = 
	\int^{\infty}_0 t^{\alpha-1} e^{-t}dt
	=
	(z+1)^{\alpha} \int^{\infty}_0 t^{\alpha-1} e^{-(z+1)t}dt.
	\]
	Hence, 
	the uniqueness theorem for holomorphic functions yields
	\[
	r_{\alpha} (z) = \frac{1}{\Gamma(\alpha)} \int^{\infty}_0 t^{\alpha-1}e^{-(z+1)t} dt
	\]
	for all $z\in \mathbb{C}_+$.
	This means that $r_{\alpha}$ is the Laplace transform of the function
	\[
	t\mapsto \frac{t^{\alpha-1} e^{-t}}{\Gamma(\alpha)},\quad t > 0.
	\]
	Hence $r_{\alpha}$ also belongs to  $\mathcal{LM} \subset \mathcal{B}$.
	
	Theorem~\ref{thm:B_calculus} (i)
	shows that the function $f_{t,\alpha} =h_t r_{\alpha}$
	satisfies
	\[
	f_{t,\alpha}(B) =  h_t(B)r_{\alpha}(B).
	\] 
	We obtain
	\[
	h_t(B) = e^{A^{-1}t}
	\]
	as shown in the proof of
	\cite[Corollary~5.7]{Batty2021}.
	By Theorem~\ref{thm:B_calculus} (ii),
	$r_{\alpha}(B)$
	coincides with the operator defined as the Hille-Phillips calculus.
	This and
	\cite[Proposition~3.3.5]{Haase2006}
	give
	\[
	r_{\alpha}(B) = (I+B)^{-\alpha} = (-A)^{-\alpha}.
	\]
	Therefore,
	\[
	f_{t,\alpha}(B) = e^{A^{-1}t}(-A)^{-\alpha}.
	\]
	Using Theorem~\ref{thm:B_calculus} (iii), we have
	\[
	\big\|e^{A^{-1}t}(-A)^{-\alpha}\big\| \leq |f_{t,\alpha}(\infty)| + 2K^2 \|f_{t,\alpha}\|_{\mathcal{B}_0} =
	2K^2 \|f_{t,\alpha}\|_{\mathcal{B}_0}.
	\]
	Thus, the desired estimate \eqref{eq:Inv_decay} holds
	by Lemma~\ref{lem:ftk_norm}.
\end{proof}

From the next example, we see that 
the norm-estimate \eqref{eq:Inv_decay}
cannot be improved in general.
\begin{example}
	\label{ex:exponential}
	Let $\gamma >0$ and
	set $\lambda_k \coloneqq -\gamma + ik$ for $k \in \mathbb{N}$.
	Define an operator $A$ on $\ell^2$ by
	\[
	(A\zeta_k)_{k \in \mathbb{N}} \coloneqq (\lambda_k \zeta_k)_{n \in \mathbb{N}}
	\]
	with domain
	\[
	D(A) \coloneqq \{ \zeta = (\zeta_k)_{k \in \mathbb{N}} \in \ell^2:  (\lambda_k \zeta_k)_{k \in \mathbb{N}} \in \ell^2  \}.
	\]
	Then
	\[
	\big\|e^{A^{-1}t} (-A)^{-\alpha}\big\| = \sup_{k \in \mathbb{N}} \frac{e^{t \re \lambda_k/ |\lambda_k|^2}}{|\lambda_k|^\alpha} =
	\sup_{k \in \mathbb{N}} \frac{e^{-\gamma t/(\gamma^2+k^2)}}{(\gamma^2+k^2)^{\alpha/2}} 
	\]
	for all $t,\alpha >0$.
	Put $f(w) \coloneqq w^{-\alpha/2}e^{-\gamma t/w}$ for $w \geq \gamma^2+1$. 
	Then
	\[
	\sup_{w \geq \gamma^2+1} f(w) =
	f\left(
	\frac{2\gamma t}{\alpha}
	\right)
	=
	\left(
	\frac{\alpha}{2e\gamma t}
	\right)^{\alpha/2}
	\]
	for all $t \geq \alpha(\gamma^2+1)/(2\gamma)$.
	From this,
	it follows that for all $k \in \mathbb{N}$ and 
	$t = \alpha(\gamma^2+k^2)/(2\gamma)$, 
	\[
	\big\|e^{A^{-1}t} (-A)^{-\alpha}\big\| = \left(
	\frac{\alpha}{2e\gamma t}
	\right)^{\alpha/2}.
	\]
	Hence 
	\[
	\liminf_{t \to \infty} 
	t^{\alpha/2} \big\|e^{A^{-1}t} (-A)^{-\alpha}\big\| \geq \left(
	\frac{\alpha}{2e\gamma}
	\right)^{\alpha/2},
	\]
	and the estimate \eqref{eq:Inv_decay}
	cannot be improved for this diagonal operator.
	Moreover, we see that the rate of 
	polynomial decay of $\|e^{A^{-1}t} (-A)^{-\alpha}\|$
	does not depend on the exponential 
	growth bound $-\gamma$ of the $C_0$-semigroup
	$(e^{tA})_{t \geq 0}$.
\end{example}

\subsection{Inverses of generators of polynomially stable semigroups}
We recall the definition of polynomially stable $C_0$-semigroups.
\begin{definition}
	A $C_0$-semigroup $(e^{At})_{t\geq 0}$
	on a Hilbert space  is {\em polynomially 
		stable with parameter $\beta >0$} if $(e^{At})_{t\geq 0}$ is 
	bounded and if
	\begin{equation}
	\label{eq:poly_estimate}
	\big\|e^{At}(I-A)^{-1}\big\| = O\left(\frac{1}{t^{1/\beta}}
	\right)\qquad (t \to \infty).
	\end{equation}
	A $C_0$-semigroup $(e^{At})_{t\geq 0}$
	is simply called  {\em polynomially 
		stable} if it is polynomially stable with some parameter $\beta>0$.
\end{definition}

Let $A$ be the generator of a polynomially stable semigroup on 
a Hilbert space $H$. 
By \cite[Theorem~1.1]{Batty2008},
$i\mathbb{R} \cap \sigma(A) = \emptyset$.
Then $A$
is invertible, 
and therefore $A^{-1}$ generates a $C_0$-semigroup on $H$.
When $A$ is normal, we 
obtain the rate of decay of $\|e^{A^{-1}t}(-A)^{-\alpha}\|$ as in the case of 
exponentially stable $C_0$-semigroups.
We also show that $\sup_{t \geq 0}\|e^{A^{-1}t}A^{-1}\| < \infty$
without assuming that $A$ is normal.

\subsubsection{Case where generators are normal}

When the generator $A$  
is a normal operator on a Hilbert space,
a spectral condition equivalent to
polynomial decay is known.
The proof can be found in \cite[Proposition~4.1]{Batkai2006}.
\begin{proposition}
	\label{prop:spectral_prop}
	Let $H$ be a Hilbert space and 
	let $A\colon D(A)\subset H \to H$ be a normal operator.
	Assume that 
	$\sigma(A) \subset \mathbb{C}_-$.
	For a fixed $\beta>0$,
	the $C_0$-semigroup $(e^{At})_{t\geq 0}$ 
	satisfies
	\begin{equation*}
	\big\|e^{At}A^{-1}\big\| = O\left(\frac{1}{t^{1/\beta}}
	\right)\qquad (t \to \infty)
	\end{equation*}
	if and only if there exist $C,\delta>0$ such that 
	\[
	|\im \lambda| \geq \frac{C}{|\re \lambda|^{1/\beta}}
	\]
	for all $\lambda \in \sigma(A)$ with $\re \lambda \geq -\delta$.
\end{proposition}

The next result gives an estimate for the rate of decay of 
$\|e^{A^{-1}t} (-A)^{-\alpha }\|$.
\begin{proposition}
	\label{prop:inv_gen_bound_normal}
	Let $H$ be a Hilbert space, and 
	let $A\colon D(A)\subset H \to H$ be  a normal operator generating
	a polynomially stable $C_0$-semigroup $(e^{At})_{t \geq 0}$ with 
	parameter $\beta>0$
	on $H$. 
	Then 
	\begin{align}
	\label{eq:Inv_decay_normal}
	\big\|e^{A^{-1}t} (-A)^{-\alpha }\big\| = 
	O\left(\frac{1}{t^{\alpha/(2+\beta)}}\right)
	\qquad  (t \to \infty)
	\end{align}
	for all $\alpha > 0$.
\end{proposition}
\begin{proof}
	For all nonzero $\lambda \in \mathbb{C}$, we obtain
	\[
	\lambda I - A^{-1} = \left(
	A - \frac{1}{\lambda } I
	\right)\lambda   A^{-1}, 
	\]
	which implies that $1/\lambda \in \sigma(A)$ if and only if $\lambda  \in \sigma(A^{-1})$.
	The normality of $A$ 
	implies 
	that of $A^{-1}$ by \cite[Theorem~5.42]{Weidmann1980}. 
	Fix $t,\alpha >0$, and
	define 
	\[
	f(\lambda) \coloneqq 
	\begin{cases}
	e^{-\lambda t}\lambda^\alpha, & \lambda 
	\in \mathbb{C} \setminus (-\infty,0], \\
	0, & \lambda = 0.
	\end{cases}
	\]
	Then $f(-A^{-1}) = e^{A^{-1}t} (-A)^{-\alpha}$.
	Moreover, $f(-A^{-1})$ is normal
	and \[
	f\big(\sigma(-A^{-1})\big) = \sigma\big(f(-A^{-1})\big);
	\] 
	see, e.g., 
	\cite[Theorem~4.5]{Haase2018}.
	Therefore,
	\begin{align*}
	\big\|e^{A^{-1}t}(-A)^{-\alpha}\big\| 
	=
	\sup_{\lambda \in \sigma(-A^{-1}) } |f(\lambda)| 
	=
	\sup_{\lambda \in \sigma(A) } \frac{e^{t \re \lambda /|\lambda|^2}}{|\lambda|^{\alpha}}.
	\end{align*}
	
	By Proposition~\ref{prop:spectral_prop},
	there exist $C,\delta >0$ such that 
	$|\im \lambda| \geq C|\re \lambda|^{-1/\beta}$ for all 
	$\lambda \in \sigma(A)$ with $|\re \lambda| \leq \delta$.
	If
	$\lambda \in \sigma(A)$ satisfies $|\re \lambda| \leq \delta$, then
	\[
	|\re \lambda| \geq \frac{C^{\beta}}{|\lambda|^{\beta}},
	\]
	and therefore
	\begin{align*}
	\frac{e^{t \re \lambda /|\lambda|^2}}{ |\lambda|^{\alpha} }
	\leq 
	\sup_{s >0} \frac{e^{-C^{\beta} t/s^{2+\beta}}}{ s^{\alpha}} 
	= 
	\left(
	\frac{\alpha}{(2+\beta)eC^{\beta}t}
	\right)^{\frac{\alpha}{2+\beta}} .
	\end{align*}
	On the other hand,
	for all $\lambda \in \sigma(A)$ with $|\re \lambda| > \delta$,
	we obtain
	\begin{align*}
	\frac{e^{t \re \lambda /|\lambda|^2}}{ |\lambda|^{\alpha} }
	\leq
	\sup_{s >0}  \frac{e^{-\delta t /s^2} }{ s^{\alpha} } 
	=
	\left(
	\frac{\alpha}{2e \delta t}
	\right)^{\frac{\alpha}{2}}.
	\end{align*}
	Thus,  the norm-estimate \eqref{eq:Inv_decay_normal} holds.
\end{proof}

As in the case of exponentially stable $C_0$-semigroups,
we present an example of an operator on $\ell^2$ showing that
the norm-estimate 
\eqref{eq:Inv_decay_normal} cannot be in general  improved.
\begin{example}
	\label{ex:poly_stable_inv}
	Set $\lambda_k \coloneqq -1/k+ ik$ for $k \in \mathbb{N}$.
	Define an operator $A$ on $\ell^2$ by
	\[
	(A\zeta_k)_{k \in \mathbb{N}} \coloneqq (\lambda_k \zeta_k)_{k \in \mathbb{N}}
	\]
	with domain
	\[
	D(A) \coloneqq \{ \zeta = (\zeta_k)_{k\in \mathbb{N}} \in \ell^2:  (\lambda_k \zeta_k)_{k \in \mathbb{N}} \in \ell^2  \}.
	\]
	By Proposition~\ref{prop:spectral_prop},
	$A$ is the generator of a polynomially stable $C_0$-semigroup with
	parameter $\beta = 1$.
	Let $t,\alpha >0$. We have 
	\[
	\big\|
	e^{A^{-1}t} (-A)^{-\alpha}
	\big\| = \sup_{k \in \mathbb{N}} \frac{e^{t \re \lambda_k/ |\lambda_k|^2}}{|\lambda_k|^\alpha} =
	\sup_{k \in \mathbb{N}} \frac{e^{-kt/(k^4+1)}}{(1/k^2+k^2)^{\alpha/2}}.
	\]
	Moreover,
	\[
	\frac{e^{-kt/(k^4+1)}}{(1/k^2+k^2)^{\alpha/2}} \geq  
	\frac{e^{-t/k^3}}{2^{\alpha /2} k^\alpha}
	\]
	for all $k \in \mathbb{N}$.
	Put $f(w) \coloneqq w^{-\alpha}e^{-t/w^3}$ for $w \geq 1$.
	Then
	\[
	\sup_{w \geq 1} f(w) =
	f\left(
	\left(\frac{3t}{\alpha}\right)^{1/3}
	\right)
	=
	\left(\frac{\alpha}{3et}\right)^{\alpha/3}
	\]
	for all $t \geq \alpha/3$.
	This implies that for all  $k \in \mathbb{N}$
	and $t = \alpha k^3 /3$, 
	\[
	\big\|e^{A^{-1}t} (-A)^{-\alpha}\big\| 
	\geq 
	\frac{1}{2^{\alpha/2}}
	\left(\frac{\alpha}{3et}\right)^{\alpha/3}.
	\]
	Therefore
	\[
	\liminf_{t \to \infty} 
	t^{\alpha/3} \big\|e^{A^{-1}t} (-A)^{-\alpha}\big\| \geq \frac{1}{2^{\alpha/2}}
	\left(\frac{\alpha}{3e}\right)^{\alpha/3},
	\]
	and the estimate \eqref{eq:Inv_decay_normal} cannot be improved
	for this  operator $A$.
\end{example}

\subsubsection{Norm-estimate through Lyapunov equations}
We show that 
if  $A$ is the generator of a polynomially stable $C_0$-semigroup
on a Hilbert space, then
$\sup_{t \geq 0}\|e^{A^{-1}t} A^{-1}\| <  \infty$.
To this end,
we use the following results on Lyapunov equations;
see \cite[Theorem~4.1.3, Corollary~6.5.1, and Theorem~6.5.2]{Curtain2020} 
for the proofs.

\begin{lemma}
	\label{lem:Lyapunov_exp}
	Let $A$ be the generator of an exponentially stable 
	semigroup $(e^{At})_{t \geq 0}$ on a Hilbert space $H$. Then
	there exists a unique self-adjoint, positive
	operator $P \in \mathcal{L}(H)$
	such that  $PD(A) \subset D(A^*)$ and 
	\begin{equation}
	\label{eq:Lyap_P_exp}
	A^*P + P A  = -I\quad \text{on $D(A)$}.
	\end{equation}
	Moreover, the operator $P$ is given by
	\begin{equation}
	\label{eq:Lyap_P_exp_int}
	Px = 
	\int^{\infty}_0 \big(e^{At} \big)^* e^{At}x dt
	\end{equation}
	for all $x \in H$.
\end{lemma}
To obtain $\sup_{t \geq 0}\|e^{A^{-1}t} A^{-1}\| <  \infty$, we shall use the fact that
$P \in \mathcal{L}(H)$ defined by
\eqref{eq:Lyap_P_exp_int} satisfies the Lyapunov equation 
\eqref{eq:Lyap_P_exp}.
\begin{lemma}
	\label{lem:Lyapunov}
	Let $A$ be the generator of a $C_0$-semigroup on a Hilbert space $H$, 
	and let 
	$C$ be a bounded linear operator from $H$ to another Hilbert space $Y$.
	If there exists a self-adjoint, non-negative operator $P \in \mathcal{L}(H)$
	such that $PD(A) \subset D(A^*)$ and
	\[
	A^*P + PA = -C^*C\quad \text{on $D(A)$},
	\]
	then
	\[
	\int^{\infty}_0 \big\|Ce^{At}x\big\|^2 dt \leq \langle x, Px \rangle
	\]
	for all $x \in H$. 
\end{lemma}

We shall need the estimate for adjoint operators in the following lemma;
see 
\cite[Lemma~2.1]{Zwart2007SF} for the proof.
\begin{lemma}
	\label{lem:bounded_inv}
	Let $A$ be the generator of 
	a bounded $C_0$-semigroup $(e^{At})_{t \geq 0}$
	on a Hilbert space $H$, and let $K \coloneqq \sup_{t \geq 0} \|e^{At}\|$.
	Then for all $x \in H$ and $\xi,\gamma >0$, 
	\begin{equation}
	\label{eq:inv_int_estimate}
	\int^{\infty}_0
	\big\|
	e^{(\gamma A-\xi I)^{-1}t}(\gamma A - \xi I)^{-1}  x
	\big\|^2 dt \leq 
	\frac{K^2 \|x\|^2}{2\xi}.
	\end{equation}
	The same estimate holds for the adjoint.
\end{lemma}

We are now in a position to prove that 
$\sup_{t \geq 0} \|e^{A^{-1}t} A^{-1}\| < \infty$ if
$A$ is the generator of a polynomially stable $C_0$-semigroup.
Actually, we will make the slightly weaker assumption that 
$\|e^{At} (I-A)^{-1}\| = O((\log t)^{-\beta})$ as $t\to \infty$ for some $\beta >1$.
Notice that if the bounded $C_0$-semigroup $(e^{At})_{t \geq 0}$ satisfies
this estimate, then $i\mathbb{R} \cap \sigma(A) = \emptyset$ 
by \cite[Theorem~1.1]{Batty2008}.
The proof of the following theorem is inspired by the arguments in the proof of 
\cite[Theorem~2.2]{Zwart2007SF}.
\begin{theorem}
	\label{thm:bounded_general_poly}
	Let $A$ be the generator of
	a bounded $C_0$-semigroup 
	on a Hilbert space $H$. If there exists $\beta >1$ such that
	\begin{equation}
	\label{eq:eA_log_bound}
	\big\|e^{At} (I-A)^{-1}\big\| = O\left(
	\frac{1}{(\log t)^{\beta}}
	\right)\qquad (t \to \infty),
	\end{equation}
	then
	\begin{equation}
	\label{eq:eA_Ainv_bound}
	\sup_{t \geq 0} \big\|e^{A^{-1}t} A^{-1}\big\| < \infty.
	\end{equation}
	In particular, 
	if $A$ is the generator  
	of a polynomially 
	stable $C_0$-semigroup on $H$, then
	the estimate \eqref{eq:eA_Ainv_bound} holds.
\end{theorem}
\begin{proof}
	The proof consists of three steps. In the first step,
	we prove an estimate analogous to \eqref{eq:inv_int_estimate} 
	for 
	bounded $C_0$-semigroups satisfying \eqref{eq:eA_log_bound}.
	In the next step, we obtain the norm-estimate of
	\[
	e^{(A-\xi_1I)^{-1}t}x - e^{(A-\xi_2I)^{-1}t} x
	\]
	for 
	$x \in D(A)$ and 
	$0< \xi_1,\xi_2 < \delta_0$, where $\delta_0 >0$ is some sufficiently small constant.
	This estimate is used  to compare 
	$e^{(A-\delta I)^{-1}t}x$ and  $e^{(A-\delta 
		e^{-N}I)^{-1}t}x$  for 
	$x \in D(A)$,
	$N \in \mathbb{N}$, and $0 < \delta < \delta_0$.
	Letting $N \to \infty$, we show 
	the desired estimate \eqref{eq:eA_Ainv_bound} in the last step.

	Step~1:
	Let $A$ be the generator of a bounded
	$C_0$-semigroup
	$(e^{At})_{t\geq 0}$ such that \eqref{eq:eA_log_bound} holds for some $\beta > 1$. 
	We write $K \coloneqq \sup_{t \geq 0} \|e^{At}\| $.
	By \eqref{eq:eA_log_bound},
	there exist $M_0 >0$ and $t_0 > 1$ such that 
	\begin{equation}
	\label{eq:log_decay}
	\big\|e^{A t}A^{-1}\big\| \leq \frac{M_0}{(\log t)^{\beta}}
	\end{equation}
	for all $t \geq t_0$. Choose $t_1 > \max\{t_0,e^{2\beta}\}$.
	There exists $M_1 >0$ such that
	\begin{equation}
	\label{eq:log_integral}
	\int_{t_1}^{\infty} \frac{e^{-\xi t}}{(\log t)^{2\beta}} dt \leq \frac{M_1}{\xi 
		|\log \xi|^{2\beta}}
	\end{equation}
	for all $0 < \xi < 1/t_1$; see \cite[Lemma~4.2]{Wakaiki2022_arXiv}.
	For $\xi >0$, define the operator $P(\xi) \in \mathcal{L}(H)$ by
	\[
	P(\xi) x \coloneqq \int^{\infty}_0 e^{-2\xi t}\big(e^{At} \big)^* e^{At}x dt,\quad 
	x \in H.
	\]
	From the estimates 
	\eqref{eq:log_decay} and \eqref{eq:log_integral}, we obtain
	\begin{align*}
	\langle x, P(\xi) x\rangle
	&=
	\int^{t_1}_0 e^{-2\xi t} \big\|e^{At}x\big\|^2 dt + 
	\int^{\infty}_{t_1} e^{-2\xi t} \big\|e^{A t}x\big\|^2 dt \\
	&\leq 
	t_1 K^2 \|x\|^2 + 
	\frac{M_0^2 M_1 \|Ax\|^2 }{2 \xi |\log (2\xi)|^{2 \beta}}
	\end{align*}
	for all $x \in D(A)$ and $0 < \xi < 1/(2t_1) \eqqcolon \delta_0$. Hence
	there exists $K_1 >0$  such that 
	\begin{equation}
	\label{eq:xi_log_xi_bound}
	\sup_{0< \xi < \delta_0}
	\xi |\log \xi|^{2\beta}
	\langle
	x,P(\xi)x
	\rangle
	\leq K_1 \|Ax\|^2
	\end{equation}
	for all $x \in D(A)$.
	
	By Lemma~\ref{lem:Lyapunov_exp},  $P(\xi)D(A) \subset D(A^*)$ and 
	\begin{equation}
	\label{eq:Lyap_step1}
	(A- \xi I)^*P(\xi) + P(\xi)(A - \xi I) = -I\quad \text{on $D(A)$}
	\end{equation}
	for all $\xi >0$.
	Since \eqref{eq:Lyap_step1} yields
	\[
	(A^*- \xi I)^{-1}P(\xi) + P(\xi)(A - \xi I)^{-1} = 
	-(A^*- \xi I)^{-1}(A - \xi I)^{-1},
	\]
	we have from
	Lemma~\ref{lem:Lyapunov} and the estimate \eqref{eq:xi_log_xi_bound} 
	that
	\begin{equation}
	\label{eq:square_int_norm1}
	\int^{\infty}_0
	\big\|
	e^{(A-\xi I)^{-1}t} (A- \xi I)^{-1}x
	\big\|^2 dt
	\leq \langle
	x,P(\xi)x
	\rangle \leq  
	\frac{K_1 \|Ax\|^2}{\xi |\log \xi|^{2\beta}}
	\end{equation}
	for all $x \in D(A)$ and $0 < \xi < \delta_0$.
	Moreover,
	by Lemma~\ref{lem:bounded_inv},
	there exists $K_2 >0$ such that 
	\begin{equation}
	\label{eq:square_int_norm2}
	\int^{\infty}_0
	\big\|
	e^{(A^*-\xi I)^{-1}t} (A^*- \xi I)^{-1}y
	\big\|^2 dt
	\leq 
	\frac{K_2  \|y\|^2}{\xi}
	\end{equation}
	for all $y \in H$ and $\xi >0$.

	Step~2:
	Take $0 < \xi_1,\xi_2 < \delta_0$.
	The variation of constants formula 
	yields
	\begin{align*}
	e^{(A-\xi_1 I)^{-1}t} x - e^{(A-\xi_2 I)^{-1}t}x  =
	\int^t_0  e^{(A-\xi_1 I)^{-1}(t-s)}  \big(
	(A-\xi_1 I)^{-1} - (A-\xi_2 I)^{-1}
	\big)  e^{(A-\xi_2 I)^{-1}s} xds
	\end{align*}
	for all $x \in H$ and $t \geq 0$.
	Moreover, the resolvent equation gives
	\[
	(A-\xi_1 I)^{-1} - (A-\xi_2 I)^{-1}
	= (\xi_1 - \xi_2) (A-\xi_1 I)^{-1} (A-\xi_2 I)^{-1}.
	\]
	Therefore,
	we obtain
	\begin{align*}
	e^{(A-\xi_1 I)^{-1}t} x - e^{(A-\xi_2 I)^{-1}t}x  =(\xi_1 - \xi_2) 
	\int^t_0  e^{(A-\xi_1 I)^{-1}(t-s)}  (A-\xi_1 I)^{-1} (A-\xi_2 I)^{-1}e^{(A-\xi_2 I)^{-1}s}xds
	\end{align*}
	for all $x \in H$ and $t \geq 0$.
	Using the estimates \eqref{eq:square_int_norm1} and \eqref{eq:square_int_norm2}, we have that 
	for all $x \in D(A)$, $y \in H$, and $t \geq 0$,
	\begin{align*}
	&\left|
	\left\langle
	y,~
	e^{(A-\xi_1 I)^{-1}t} x - e^{(A-\xi_2 I)^{-1}t} x
	\right\rangle
	\right| \\
	&\quad =
	|\xi_1 - \xi_2|\,
	\bigg|
	\int^t_0
	\Big\langle
	e^{(A^*-\xi_1 I)^{-1}(t-s)} (A^*- \xi_1 I)^{-1}y,\,
	e^{(A-\xi_2 I)^{-1}s} (A- \xi_2 I)^{-1}x
	\Big\rangle
	ds
	\bigg| \\
	&\quad \leq
	|\xi_1 - \xi_2|
	\left(
	\int^{\infty}_0
	\big\|
	e^{(A^*-\xi_1 I)^{-1}t} (A^*- \xi_1 I)^{-1}y
	\big\|^2 dt
	\right)^{1/2}  \left(
	\int^{\infty}_0
	\big\|
	e^{(A-\xi_2 I)^{-1}t} (A- \xi_2 I)^{-1}x
	\big\|^2 dt
	\right)^{1/2} \\
	&\quad \leq
	\frac{ K_0|\xi_1 - \xi_2| }{ 
		\sqrt{\xi_1\xi_2} \, |\log \xi_2|^{\beta}} \|Ax\| \, \|y\|,
	\end{align*}
	where $K_0 \coloneqq \sqrt{K_1K_2}$.
	Hence
	\begin{equation}
	\label{eq:inv_diff_bound}
	\big\|
	e^{(A-\xi_1 I)^{-1}t} x - e^{(A-\xi_2 I)^{-1}t} x 
	\big\|
	\leq 
	\frac{K_0|\xi_1 - \xi_2| }{ \sqrt{\xi_1\xi_2} \, |\log \xi_2|^{\beta} }\|Ax\|
	\end{equation}
	for all $x \in D(A)$ and $t \geq 0$.
	
	Step~3:
	Let $N \in \mathbb{N}$ and $0<\delta < \delta_0\, (< 1)$.
	Substituting $\xi_1 = \delta e^{-n+1}$ and $\xi_2 = \delta e^{-n}$, $n=1,2,\dots,N$, 
	into the estimate \eqref{eq:inv_diff_bound},
	we obtain
	\begin{align*}
	\big\|
	e^{(A- \delta I)^{-1}t} x  - e^{(A- \delta e^{-N} I)^{-1}t} x 
	\big\|
	& =
	\left\|
	\sum_{n=1}^N
	e^{(A- \delta e^{-n+1} I)^{-1}t} x - e^{(A-\delta e^{-n}  I)^{-1}t} x 
	\right\| \\
	& \leq 
	K_0 \|Ax\| \sum_{n=1}^N
	\frac{|e^{-n+1} - e^{-n}| }{\sqrt{e^{-n+1} e^{-n}}\, |\log (\delta  e^{-n})|^{\beta}} \\
	&\leq 
	\frac{(e-1)K_0}{\sqrt{e}  } 
	\|Ax\|\sum_{n=1}^{N} \frac{1}{|n - \log \delta |^{\beta}}
	\end{align*}
	for all $x \in D(A)$ and  $t \geq 0$.
	From the assumption $\beta > 1$, we obtain
	\[
	K_3 \coloneqq \sum_{n=1}^{\infty} \frac{1}{|n - \log \delta |^{\beta}} < \infty.
	\]
	Since
	\[
	\lim_{N \to \infty}
	\big\|e^{(A-\delta e^{-N}  I)^{-1}t} - e^{A^{-1}t}  \big\| = 0
	\]
	for each $t \geq 0$,
	it follows that
	\[
	\big\|
	e^{(A- \delta I)^{-1}t} x  - e^{A^{-1}t} x 
	\big\| \leq 
	\frac{(e-1)K_0K_3}{\sqrt{e} } 
	\|Ax\|
	\]
	for all $x \in D(A)$ and $t \geq 0$.
	This gives
	\[
	\sup_{t\geq 0}
	\big\|
	e^{(A-\delta  I)^{-1}t} A^{-1}  - e^{A^{-1}t} A^{-1} 
	\big\| < \infty.
	\]
	Since $A- \delta I$ generates an exponentially stable $C_0$-semigroup,
	Theorem~\ref{thm:inv_gen_bound} implies that 
	\[
	\sup_{t \geq 0 }\big\|e^{(A- \delta I)^{-1}t} A^{-1}\big\| < \infty.
	\]
	Thus,
	\begin{align*}
	\sup_{t \geq 0}
	\big\|
	e^{A^{-1}t} A^{-1}
	\big\|\leq \sup_{t \geq 0}
	\big\|
	e^{(A-\delta I)^{-1}t} A^{-1}
	\big\| + 
	\sup_{t \geq 0}\big\|
	e^{(A- \delta I)^{-1}t} A^{-1} - 
	e^{A^{-1}t} A^{-1}
	\big\| < \infty.\qedhere
	\end{align*}
\end{proof}

\section{Crank-Nicolson Scheme with Smooth Initial Data}
\label{sec:CNS}
Let $A$ be the generator of a bounded $C_0$-semigroup  on 
a Hilbert space $H$.  
For $\tau > 0$, we define
\begin{equation}
\label{eq:Ad_def}
A_d(\tau) \coloneqq \left(
I + \frac{\tau}{2}A
\right) \left(
I - \frac{\tau}{2}A
\right)^{-1}. 
\end{equation}
Let $(\tau_n)_{n \in \mathbb{N}_0}$ be a 
sequence of strictly positive real numbers.
We consider the time-varying difference equation
\begin{equation}
\label{eq:difference_equation}
x_{n+1} = A_d(\tau_n) x_n,\quad n \in \mathbb{N}_0;\qquad  x_0 \in H.
\end{equation}
In this section, we study the decay rate of the solution 
$(x_n)_{n \in \mathbb{N}_0}$ of the difference equation 
\eqref{eq:difference_equation} with smooth initial data.

\subsection{Generators of exponentially stable semigroups}
Let $H$ be a Hilbert space, and
let  $A\colon D(A) \subset H \to H$ 
be injective.
Suppose that 
$A$ and $A^{-1}$
generate bounded $C_0$-semigroups on $H$. 
Take $\xi >0$.
By Lemma~\ref{lem:Lyapunov_exp},
there  exist unique self-adjoint, positive
operators $P(\xi), Q(\xi) \in \mathcal{L}(H)$
such that  $P(\xi)D(A) \subset D(A^*)$, $Q(\xi)D(A^{-1}) \subset D((A^{-1})^*)$,
and 
\begin{subequations}
	\label{eq:PQ_Lyap}
	\begin{align}
	(A- \xi I)^*P(\xi) + P(\xi)(A - \xi I) &= -I\quad \text{on $D(A)$}, \\
	(A^{-1} - \xi I)^*Q(\xi) + Q(\xi)(A^{-1} - \xi I) &= -I\quad \text{on $D(A^{-1})$}.
	\end{align}
\end{subequations}
The operators $P(\xi)$ and $Q(\xi)$ are given by
\begin{subequations}
	\label{eq:PQ_def}
	\begin{align}
	P(\xi) x &= 
	\int^{\infty}_0 e^{-2\xi t} (e^{At})^* e^{At}x dt, \label{eq:P_def}\\
	Q(\xi) x &= 
	\int^{\infty}_0 e^{-2\xi t} (e^{A^{-1}t})^* e^{A^{-1}t}x dt
	\label{eq:Q_def}
	\end{align}
\end{subequations}
for all $x \in X$. 
To estimate the decay rate of the solution $(x_n)_{n \in \mathbb{N}_0}$ of the difference equation 
\eqref{eq:difference_equation},
we shall use that the operators
$P(\xi)$ and $Q(\xi)$ defined as in \eqref{eq:PQ_def}
solve the Lypapunov equations \eqref{eq:PQ_Lyap}.

In the proofs of \cite[Lemmas~2.1 and 2.2]{Piskarev2007},
the following result has been obtained from the Lypapunov equations \eqref{eq:PQ_Lyap}.
\begin{lemma}
	\label{lem:PZ_lemma}
	Let $A$ be the generator of a bounded $C_0$-semigroup on 
	a Hilbert space $H$. 
	Suppose that $A$ is injective and that $A^{-1}$ also generate a bounded $C_0$-semigroup on  $H$. 
	Let
	$0 < \tau_{\min} \leq \tau_n \leq \tau_{\max}
	< \infty$ for all $n \in \mathbb{N}_0$, and 
	let 
	$P(\xi), Q(\xi) \in \mathcal{L}(H)$ be given by \eqref{eq:PQ_def}
	for $\xi >0$. Define
	$R(r) \in \mathcal{L}(H)$ by 
	\[
	R(r) \coloneqq \frac{2}{\tau_{\min}}P\left(
	\frac{\xi_r}{\tau_{\max}}
	\right) + 
	2\tau_{\max} Q(\tau_{\min} \xi_r),\quad r \in (0,1),
	\]
	where
	\[
	\xi_r \coloneqq \frac{1-r^2}{2(r^2+1)}.
	\]
	Then
	there exists a constant $M >0$ such that 
	the solution $(x_n)_{n \in \mathbb{N}_0}$ of the difference equation
	\eqref{eq:difference_equation} satisfies
	\[
	|(n+1)r^n \langle y, x_n
	\rangle| \leq 
	\frac{M \|y\| }{\sqrt{1-r}} \sqrt{\langle 
		x_0, R(r)x_0
		\rangle}
	\] 
	for all $x_0, y \in X$, $n \in \mathbb{N}$, and $r \in (0,1)$.
\end{lemma}

Now we estimate $\langle x,P(\xi) x \rangle$ and 
$\langle x,Q(\xi) x \rangle$, by using the integral
representations
\eqref{eq:PQ_def}.
Suppose that $A$ is the generator of an exponentially stable 
$C_0$-semigroup $(e^{At})_{t \geq 0}$.
Let
$K_0 \geq 1$ and $\omega >0$ satisfy
$\|e^{At}\| \leq K_0e^{-\omega t}$ for all $t \geq 0$. Then
the operator $P(\xi)$ given by \eqref{eq:P_def} satisfies
\begin{align*}
\langle
x, P(\xi)x 
\rangle
\leq 
K_0^2\|x\|^2 \int^{\infty}_0 e^{-2(\omega + \xi)t}dt   
\leq  \frac{K_0^2 \|x\|^2}{2\omega } 
\end{align*}
for all $x \in X$ and  $\xi >0$. Hence 
\begin{equation}
\label{eq:P_bound}
\sup_{\xi >0}
\langle
x, P(\xi)x 
\rangle \leq K_1 \|x\|^2
\end{equation}
for all $x \in X$, 
where $K_1 \coloneqq K_0^2/(2\omega)$.
Using the 
norm-estimate \eqref{eq:Inv_decay} 
for $e^{A^{-1}t}(-A)^{-\alpha}$ with $0 < \alpha \leq 1$,
we estimate $\langle
x, Q(\xi)x
\rangle$  for $x \in D((-A)^{\alpha})$ in the next lemma.
\begin{lemma}
	\label{eq:Q_bound_exp}
	Let $A$ be the generator of 
	an exponentially stable $C_0$-semigroup $(e^{At})_{t \geq 0}$
	on a Hilbert space $H$.
	Suppose that $A^{-1}$ generates
	a  bounded $C_0$-semigroup 
	$(e^{A^{-1}t})_{t \geq 0}$
	on  $H$. 
	Then 
	$Q(\xi) \in \mathcal{L}(H)$ given by \eqref{eq:Q_def} for $\xi >0$
	satisfies the following statements:
	\begin{enumerate}
		\renewcommand{\labelenumi}{(\roman{enumi})}
		\item
		For each $0< \alpha< 1$,
		there exists $K>0$ such that 
		\begin{equation}
		\label{eq:Q_bound1}
		\sup_{0< \xi < 1}
		\xi^{1-\alpha}
		\langle
		x, Q(\xi)x
		\rangle
		\leq 
		K \|(-A)^{\alpha}x\|^2
		\end{equation}
		for all $x \in D((-A)^{\alpha})$
		\item
		There exists  $K>0$ such that
		\begin{equation}
		\label{eq:Q_bound2}
		\sup_{0< \xi < 1/2}
		\frac{
			\langle
			x, Q(\xi)x
			\rangle
		}{\log(1/\xi)}
		\leq 
		K \|Ax\|^2
		\end{equation}
		for all $x \in D(A)$.
	\end{enumerate}
\end{lemma}
\begin{proof}
	Let $0 < \alpha \leq 1$ and 
	$
	M_1 \coloneqq 
	\sup_{t \geq 0}\|e^{A^{-1} t}\| < \infty$.
	By Theorem~\ref{thm:inv_gen_bound},
	there exist $M_2, t_0 >0$ such that 
	\[
	\big\|e^{A^{-1}t}x\big\| \leq \frac{M_2 \|(-A)^{\alpha}x\|}{t^{\alpha/2}}
	\] 
	for all $x \in D((-A)^{\alpha})$ and $t \geq t_0$.
	Hence 
	\begin{align*}
	\langle x, Q(\xi) x\rangle 
	&=
	\int^{t_0}_0 e^{-2\xi t} \big\|e^{A^{-1}t}x\big\|^2 dt + 
	\int^{\infty}_{t_0} e^{-2\xi t}\big\|e^{A^{-1}t}x\big\|^2 dt \\
	&\leq 
	t_0 M_1^2 \|x\|^2 + 
	M_2^2 \|(-A)^{\alpha}x\|^2 \int^{\infty}_{t_0} \frac{e^{-2\xi t} }{t^{\alpha}} dt
	\end{align*}
	for all $x \in D((-A)^{\alpha})$ and $\xi >0$. 
	When $0 < \alpha < 1$, we have
	\[
	\int^{\infty}_{t_0} \frac{e^{-2\xi t} }{t^{\alpha}} dt \leq 
	\int^{\infty}_0
	\frac{e^{-2\xi t} }{t^{\alpha}} dt =
	\frac{\Gamma(1-\alpha)}{(2\xi)^{1-\alpha}}.
	\]
	Hence there is $K>0$ such that 
	\eqref{eq:Q_bound1} holds for all $x \in D((-A)^{\alpha})$.
	
	Next we consider the case $\alpha = 1$. The exponential integral
	satisfies
	\[
	\int_\tau^{\infty} \frac{e^{-t}}{t} dt
	\leq 
	e^{-\tau}\log\left( 1+ \frac{1}{\tau} \right)
	\]
	for all $\tau >0$; see \cite[inequality (5)]{Gautschi1959}.
	Therefore,
	\[
	\int_{t_0}^{\infty}
	\frac{e^{-2\xi  t}}{t} dt =
	\int_{2\xi  t_0}^{\infty}
	\frac{e^{-t}}{t} dt \leq e^{-2\xi  t_0} \log\left ( 1 + \frac{1}{2\xi  t_0} \right)
	\]
	for all  $\xi >0$.
	There exists $c>0$ such that for all $0 < \xi <1/2$,
	\[
	\log\left(
	1 + \frac{1}{2\xi t_0}
	\right) \leq  c\log\left( \frac{1}{\xi }
	\right).
	\]
	Thus, 
	\eqref{eq:Q_bound2} holds for all $x \in D(A)$ and some suitable constant $K >0$.
\end{proof}

Combining Lemma~\ref{lem:PZ_lemma} with 
the estimates \eqref{eq:P_bound}--\eqref{eq:Q_bound2},
we obtain estimates for the decay rate of the solution 
$(x_n)_{n \in \mathbb{N}_0}$ of the difference equation 
\eqref{eq:difference_equation}
with smooth initial data.
\begin{theorem}
	\label{thm:CN_exp_stable}
	Let $A$ be the generator of 
	an exponentially stable $C_0$-semigroup $(e^{At})_{t \geq 0}$
	on a Hilbert space $H$.
	Suppose that $A^{-1}$ generates
	a  bounded $C_0$-semigroup 
	$(e^{A^{-1}t})_{t \geq 0}$
	on  $H$. 
	If $0 < \inf_{n \in \mathbb{N}_0} \tau_n \leq \sup_{n \in \mathbb{N}_0} \tau_n
	< \infty$, then there exists $K>0$ such that
	the solution $(x_n)_{n \in \mathbb{N}_0}$ of the difference equation 
	\eqref{eq:difference_equation} satisfies
	the following statements:
	\begin{enumerate}
		\renewcommand{\labelenumi}{(\roman{enumi})}	
		\item
		For each $0 < \alpha < 1$,
		there exists $K>0$ such that
		\begin{equation}
		\label{eq:xn_estimate1}
		\|x_n\| \leq 
		\frac{K}{n^{\alpha/2}}
		\|(-A)^\alpha x_0\|
		\end{equation}
		for all $x_0 \in D((-A)^{\alpha})$ and $n \geq 1$.
		\item
		There exists $K>0$ such that
		\begin{equation}
		\label{eq:xn_estimate2}
		\|x_n\| \leq K
		\sqrt{\frac{\log n}{n}}
		\|Ax_0\|
		\end{equation}
		for all $x_0 \in D(A)$ and $n \geq 2$.
	\end{enumerate}
\end{theorem}
\begin{proof}
	Set 
	\[
	\tau_{\min} \coloneqq \inf_{n \in \mathbb{N}_0} \tau_n >0,\quad 
	\tau_{\max} \coloneqq \sup_{n \in \mathbb{N}_0} \tau_n < \infty,
	\]
	and let
	the operator $R(r) \in \mathcal{L}(H)$ be as in Lemma~\ref{lem:PZ_lemma}.
	Then there is $M>0$ such that 
	the solution $(x_n)_{n \in \mathbb{N}_0}$ of the difference equation 
	\eqref{eq:difference_equation} satisfies
	\begin{equation}
	\label{eq:xn_bound}
	\|x_n\| \leq  \frac{M \sqrt{\langle x_0, R(r)x_0 \rangle } }{(n+1)r^n \sqrt{1-r}}
	\end{equation}
	for all $x_0 \in X$, $n \in \mathbb{N}$, and $r \in (0,1)$.
	
	Let $0 < \alpha \leq 1$, and
	define
	\[
	f_{\alpha}(\xi) \coloneqq
	\begin{cases}
	\xi^{\alpha-1}, & 0<\alpha<1, \vspace{5pt}\\
	\log \left( \dfrac{1}{\xi} \right), & \alpha = 1
	\end{cases}
	\]
	for $0 < \xi < 1$.
	By the estimates \eqref{eq:P_bound}--\eqref{eq:Q_bound2}, there exist
	$K_1,K_2 >0$ and $r_0\in(0,1)$ such that 
	\begin{equation}
	\label{xRx_bound}
	\langle x_0, R(r)x_0 \rangle
	\leq
	\frac{2K_1}{\tau_{\min}} \|x_0\|^2 + 2\tau_{\max} K_2 \|(-A)^{\alpha}x_0\|^2f_{\alpha}(\tau_{\min} \xi_r)
	\end{equation}
	for all $x_0 \in D((-A)^\alpha)$ and  $r \in (r_0,1)$.
	Put $r = n/(n+1)$ for $n \in \mathbb{N}$.
	Then
	\[
	\xi_r = \frac{1-r^2}{2(r^2+1)} =  \frac{2n+1}{2(2n^2+2n+1)}
	\]
	and
	\[
	\frac{1}{(n+1)r^n \sqrt{1-r}} = \frac{(1+1/n)^n}{\sqrt{n+1}}.
	\]
	Combining the 
	estimates \eqref{eq:xn_bound} and 
	\eqref{xRx_bound}, we have that
	there exist $K_4 >0$ and $n_0 \in \mathbb{N}$ such that 
	\[
	\|x_n\| \leq K_4 \sqrt{\frac{f_{\alpha} (1/n)}{n}}
	\|(-A)^{\alpha}x_0\|
	\]
	for all $x_0 \in D((-A)^{\alpha})$ and $n \geq  n_0 + 1$.
	By definition,
	\[
	\sqrt{\frac{f_{\alpha} (1/n)}{n}} =
	\begin{cases}
	\dfrac{1}{n^{\alpha/2}},  & 0<\alpha<1, \vspace{6pt}\\
	\sqrt{\dfrac{\log n }{n}},  & \alpha = 1
	\end{cases}
	\]
	for all $n \geq 2$.
	Since
	\[
	\|A_d(\tau)\| = 
	\left\|
	2\left(
	I - \frac{\tau}{2}A
	\right)^{-1}
	-
	I
	\right\| \leq c
	\]
	for all $\tau \in [\tau_{\min},\tau_{\max}]$ and some $c \geq 1$, it follows that
	$\|x_n\| \leq c^{n_0} \|x_0\|$ for all $x_0 \in X$ and $0\leq n \leq n_0$.
	Thus, we obtain the desired conclusion.
\end{proof}

We compare the norm-estimates \eqref{eq:xn_estimate1} and \eqref{eq:xn_estimate2}
with those in the time-invariant case $\tau_n \equiv 2$, using a simple example.
\begin{example}
	Let the operator $A$ on $\ell^2$ be as in Example~\ref{ex:exponential}.
	Define $A_d \coloneqq A_d(2)= (I+A)(I-A)^{-1}$ and let $0 < \alpha \leq 1$.
	To obtain the decay rate of 
	the solution $(x_n)_{n \in \mathbb{N}_0}$ of 
	the time-invariant difference equation 
	\[
	x_{n+1} = A_d x_n,\quad n \in \mathbb{N}_0;\qquad 
	x_0 \in D((-A)^{\alpha}),
	\]
	we
	estimate $\|A_d^{n}(-A)^{-\alpha}\|$ for $n \in \mathbb{N}_0$.
	
	We have that for all $n \in \mathbb{N}_0$,
	\[
	\big\|A_d^{n^2}(-A)^{-\alpha}\big\| = 
	\sup_{k \in \mathbb{N}} \left|
	\frac{1+\lambda_k}{1-\lambda_k}
	\right|^{n^2} |\lambda_k|^{-\alpha} \geq 
	\left|\frac{1+\lambda_n}{1-\lambda_n}
	\right|^{n^2} |\lambda_n|^{-\alpha}.
	\]
	Moreover, 
	\begin{align*}
	\left|
	\frac{1+\lambda_n}{1-\lambda_n}
	\right|^{n^2}
	\to e^{-2\gamma },\quad 
	n^{\alpha} |\lambda_n|^{-\alpha} \to 1
	\end{align*}
	as $n \to \infty$.
	Hence
	\[
	\liminf_{n \to \infty} 
	n^{\alpha} \|A_d^{n^2} (-A)^{-\alpha}\| \geq e^{-2\gamma}.
	\]
	This implies that 
	the norm-estimate \eqref{eq:xn_estimate1} for the case $0 < \alpha <1$
	is optimal in the sense that 
	one cannot obtain any better rates in general.

	Next we show that 
	\begin{equation}
	\label{eq:decay_example_exp}
	\|A_d^{n}A^{-1}\| = O\left(\frac{1}{\sqrt{n}} \right)\qquad  (n \to \infty). 
	\end{equation}
	For all $n \in \mathbb{N}$,
	\[
	\|A_d^{n}A^{-1}\| = 
	\sup_{k \in \mathbb{N}} \left|
	\frac{1+\lambda_k}{1-\lambda_k}
	\right|^{n} |\lambda_k|^{-1} 
	\leq \sqrt{\sup_{w \geq \gamma^2 + 1} f_n(w)},
	\]
	where
	\[
	f_n(w) \coloneqq \left(
	\frac{w+1-2\gamma}{w+1+2\gamma}
	\right)^n \frac{1}{w}.
	\]
	A simple calculation shows that for all sufficiently large $n \in \mathbb{N}$,
	\[
	\sup_{w \geq \gamma^2 + 1} f_n(w) = f_n(w_n),
	\] 
	where 
	\[
	w_n \coloneqq 2\sqrt{\gamma^2n^2+\gamma^2-\gamma n}  + 2\gamma n - 1.
	\]
	Since
	\[
	\left(
	\frac{w_n+1-2\gamma}{w_n+1+2\gamma}
	\right)^n \to e^{-1},\quad 
	\frac{n}{w_n} \to \frac{1}{4\gamma}
	\]
	as $n \to \infty$, we conclude that 
	the estimate \eqref{eq:decay_example_exp} holds.
	It is not clear, in general, whether one can remove the logarithmic term $\sqrt{\log n}$ in 
	the estimate \eqref{eq:xn_estimate2} for the case $\alpha =1$.
\end{example}

\subsection{Normal generators of polynomially stable semigroups}
Suppose that $A$ is a normal
operator on a Hilbert space $H$ and  generates a polynomially stable $C_0$-semigroup. As in the case of exponentially stable $C_0$-semigroups,
one can obtain an estimate for the decay rate of the solution 
$(x_n)_{n \in \mathbb{N}_0}$ of the difference equation 
\eqref{eq:difference_equation} with smooth initial data.
Note that $A^{-1}$ generates a contraction $C_0$-semigroup on $H$; see \cite[Lemma~4.1]{Zwart2007}.

The following lemma gives 
estimates analogous to those in Lemma~\ref{eq:Q_bound_exp}.
\begin{lemma}
	\label{lem:normal_Q_bound}
	Let $H$ be a Hilbert space and 
	let $A\colon D(A)\subset H \to H$ be  a normal operator generating
	a polynomially stable $C_0$-semigroup $(e^{At})_{t \geq 0}$ with 
	parameter $\beta>0$
	on $H$. Then 
	$Q(\xi) \in \mathcal{L}(H)$ given by \eqref{eq:Q_def} for $\xi >0$
	satisfies the following statements:
	\begin{enumerate}
		\renewcommand{\labelenumi}{(\roman{enumi})}
		\item
		For each $0< \alpha < 1+\beta/2$,
		there exists $K>0$ such that 
		\begin{equation}
		\label{eq:Q_bound_normal1}
		\sup_{0< \xi  < 1}
		\xi^{1-\frac{2\alpha}{2+\beta}}
		\langle
		x, Q(\xi)x
		\rangle
		\leq 
		K \|(-A)^{\alpha}x\|^2
		\end{equation}
		for all $x \in D((-A)^{\alpha})$.
		\item
		There exists  $K>0$ such that
		\begin{equation}
		\label{eq:Q_bound_normal2}
		\sup_{0< \xi < 1/2}
		\frac{
			\langle
			x, Q(\xi)x
			\rangle
		}{\log(1/\xi)}
		\leq 
		K \|(-A)^{1+\beta/2}x\|^2
		\end{equation}
		for all $x \in D((-A)^{1+\beta/2})$.
	\end{enumerate}
\end{lemma}
\begin{proof}
	Let $0< \alpha \leq 1+\beta/2$, and
	put
	\[
	\widetilde \alpha \coloneqq \frac{2\alpha}{2+\beta} \leq 1.
	\]
	By Proposition~\ref{prop:inv_gen_bound_normal},
	there exist $M, t_0 >0$ such that 
	\begin{equation}
	\label{eq:inv_gene_poly_bound}
	\big\|e^{A^{-1}t}x\big\| \leq \frac{M \|(-A)^{\alpha}x\|}{t^{\widetilde \alpha / 2}}
	\end{equation}
	for all $x \in D((-A)^{\alpha})$ and  $t \geq t_0$.
	The rest of the proof 
	is quite similar to that of Lemma~\ref{eq:Q_bound_exp},
	and hence we omit it.
\end{proof}

If $A$ is the generator of
a polynomially stable $C_0$-semigroup $(e^{At})_{t \geq 0}$ with
parameter $\beta>0$, then \cite[Proposition~3.1]{Batkai2006} shows that
for all $\alpha >0$, there exist $M,t_0 >0$ such that
\[
\big\|e^{A t}x\big\| \leq \frac{M \|(-A)^{\alpha} x\|}{t^{\alpha/\beta}}
\]
for all $x \in D((-A)^{\alpha})$ and  $t \geq t_0$. The decay rate $t^{-\alpha/\beta}$
is faster than the decay rate of $\|e^{A^{-1}t}x\| $  in  \eqref{eq:inv_gene_poly_bound}. Hence, the operator $P(\xi)$ given by \eqref{eq:P_def} satisfies 
the same estimates as $Q(\xi)$.

Using the estimates on $P(\xi)$ and $Q(\xi)$, we derive
a norm-estimate for the solution $(x_n)_{n \in \mathbb{N}_0}$ 
of the difference equation
\eqref{eq:difference_equation}.
Since this result can be obtained by the same arguments as 
in the proof of Theorem~\ref{thm:CN_exp_stable},
we omit the proof. 

\begin{proposition}
	\label{prop:normal_decay_rate}
	Let $H$ be a Hilbert space and 
	let $A\colon D(A)\subset H \to H$ be  a normal operator generating
	a polynomially stable $C_0$-semigroup $(e^{At})_{t \geq 0}$ with 
	parameter $\beta>0$
	on $H$. 
	If $0 < \inf_{n \in \mathbb{N}_0} \tau_n \leq \sup_{n \in \mathbb{N}_0} \tau_n
	< \infty$, then the solution $(x_n)_{n \in \mathbb{N}_0}$  of
	the difference equation \eqref{eq:difference_equation} 
	satisfies
	the following statements:
	\begin{enumerate}
		\renewcommand{\labelenumi}{(\roman{enumi})}	
		\item
		For each $0 < \alpha < 1 + \beta/2$,
		there exists $K>0$ such that
		\begin{equation}
		\label{eq:decay_sol_b_small}
		\|x_n\| \leq 
		\frac{K}{n^{\alpha/(2+\beta)}}
		\|(-A)^\alpha x_0\|
		\end{equation}
		for all $x_0 \in D((-A)^{\alpha})$ and $n \geq 1$.
		\item
		There exists $K>0$ such that
		\begin{equation}
		\label{eq:decay_sol_b_large}
		\|x_n\| \leq K
		\sqrt{\frac{\log n}{n}}
		\|(-A)^{1+\beta/2}x_0\|
		\end{equation}
		for all $x_0 \in D((-A)^{1+\beta/2})$ and $n \geq 2$.
	\end{enumerate}
\end{proposition}

\begin{example}
	Consider again the operator $A$ on $\ell^2$ 
	in Example~\ref{ex:poly_stable_inv}. We recall that
	the parameter $\beta$ of polynomial decay is given by $\beta = 1$.
	Define $A_d \coloneqq A_d(2)= (I+A)(I-A)^{-1}$.
	It has been shown in \cite[Example~4.6]{Wakaiki2021JEE}
	that 
	the estimate $\|A_d^n (-A)^{-3/m}\| = O(n^{-1/m})$ is optimal for all $m \in \mathbb{N}$.
	Hence we see from the case $m =3$ 
	that the estimate \eqref{eq:decay_sol_b_small} cannot in general be
	improved.
	Although
	$\|A_d^n (-A)^{-3/2}\| = O(1/\sqrt{n})$ can be deduced from 
	the case $m=2$, it is open whether 
	the logarithmic term $\sqrt{\log n}$ in the 
	norm-estimate \eqref{eq:decay_sol_b_large}
	can be omitted.
\end{example}


\begin{thebibliography}{10}
	
	\bibitem{Azizov2004}
	T.~Ya. Azizov, A.~I. Barsukov, and A.~Dijksma.
	\newblock {Decompositions of a Krein space in regular subspaces invariant under
		a uniformly bounded $C_0$-semigroup of bi-contractions}.
	\newblock {\em J. Funct. Anal.}, 211:324--354, 2004.
	
	\bibitem{Bakarev2001}
	N.~Bakaev and A.~Ostermann.
	\newblock Long-term stability of variable stepsize approximations of
	semigroups.
	\newblock {\em Math. Comp.}, 71:1545--1567, 2001.
	
	\bibitem{Batkai2006}
	A.~B\'atkai, K.-J. Engel, J.~Pr\"uss, and R.~Schnaubelt.
	\newblock Polynomial stability of operator semigroups.
	\newblock {\em Math. Nachr.}, 279:1425--1440, 2006.
	
	\bibitem{Batty2008}
	C.~J.~K. Batty and T.~Duyckaerts.
	\newblock {Non-uniform stability for bounded semi-groups on Banach spaces}.
	\newblock {\em J. Evol. Equations}, 8:765--780, 2008.
	
	\bibitem{Batty2021}
	C.~J.~K. Batty, A.~Gomilko, and Yu. Tomilov.
	\newblock {A Besov algebra calculus for generators of operator semigroups and
		related norm-estimates}.
	\newblock {\em Math. Ann.}, 379:23--93, 2021.
	
	\bibitem{Brenner1979}
	P.~Brenner and V.~Thom\'ee.
	\newblock On rational approximations of semigroups.
	\newblock {\em SIAM J. Numer. Anal.}, 16:683--694, 1979.
	
	\bibitem{Crouzeix1993}
	M.~Crouzeix, S.~Larsson, S.~Piskarev, and V.~Thom\'ee.
	\newblock The stability of rational approximations of analytic semigroups.
	\newblock {\em BIT}, 33:74--84, 1993.
	
	\bibitem{Curtain2020}
	R.~F. Curtain and H.~J. Zwart.
	\newblock {\em An Introduction to Infinite-Dimensional Systems: A State Space
		Approach}.
	\newblock New York: Springer, 2020.
	
	\bibitem{deLaubenfels1988}
	R.~deLaubenfels.
	\newblock Inverses of generators.
	\newblock {\em Proc. Amer. Math. Soc.}, 104:443--448, 1988.
	
	\bibitem{deLaubenfels2009}
	R.~deLaubenfels.
	\newblock Inverses of generators of nonanalytic semigroups.
	\newblock {\em Studia Math.}, 191:11--38, 2009.
	
	\bibitem{Fackler2016}
	S.~Fackler.
	\newblock {A short counterexample to the inverse generator problem on
		non-Hilbertian reflexive $L^p$-spaces}.
	\newblock {\em Arch. Math.}, 106:383--389, 2016.
	
	\bibitem{Gautschi1959}
	W.~Gautschi.
	\newblock Some elementary inequalities relating to the gamma and incomplete
	gamma function.
	\newblock {\em J. Math. Phys.}, 38:77--81, 1959.
	
	\bibitem{Gomilko2004}
	A.~Gomilko.
	\newblock {Cayley transform of the generator of a uniformly bounded
		$C_0$-semigroup of operators}.
	\newblock {\em Ukrainian Math. J.}, 56:1212--1226, 2004.
	
	\bibitem{Gomilko2017}
	A.~Gomilko.
	\newblock Inverses of semigroup generators: a survey and remarks.
	\newblock In {\em \'Etudes op\'eratorielles, Banach Center Publ., vol. 112,
		Polish Acad. Sci., Warsaw}, pages 107--142, 2017.
	
	\bibitem{Gomilko2011}
	A.~Gomilko, H.~Zwart, and N.~Besseling.
	\newblock Growth of semigroups in discrete and continuous time.
	\newblock {\em Studia Math.}, 206:273--292, 2011.
	
	\bibitem{Gomilko2007MS}
	A.~Gomilko, H.~Zwart, and Yu. Tomilov.
	\newblock {Inverse operator of the generator of a $C_0$-semigroup}.
	\newblock {\em Sb. Math.}, 198:1095--1110, 2007.
	
	\bibitem{Guo2006}
	B.-Z. Guo and H.~Zwart.
	\newblock {On the relation between stability of continuous- and discrete-time
		evolution equations via the Cayley transform}.
	\newblock {\em Integral Equations Operator Theory}, 54:349--383, 2006.
	
	\bibitem{Haase2006}
	M.~Haase.
	\newblock {\em The Functional Calculus for Sectorial Operators}.
	\newblock Basel: Birkh\"auser, 2006.
	
	\bibitem{Haase2018}
	M.~Haase.
	\newblock {\em Lectures on Functional Calculus}.
	\newblock 21st International Internet Seminar, Kiel Univ, 2018.
	Available at \url{https://www.math.uni-kiel.de/isem21/en/course/phase1/
		isem21-lectures-on-functional-calculus}
	
	\bibitem{Komatsu1966}
	H.~Komatsu.
	\newblock Fractional powers of operators.
	\newblock {\em Pacific J. Math.}, 19:285--346, 1966.
	
	\bibitem{Palencia1993}
	C.~Palencia.
	\newblock {A stability result for sectorial operators in Banach spaces}.
	\newblock {\em SIAM J. Numer. Anal.}, 30:1373--1384, 1993.
	
	\bibitem{Phillips1959}
	R.~S. Phillips.
	\newblock Dissipative operators and hyperbolic systems of partial differential
	equations.
	\newblock {\em Trans. Amer. Math. Soc.}, 90:193--254, 1959.
	
	\bibitem{Piskarev2007}
	S.~Piskarev and H.~Zwart.
	\newblock {Crank-Nicolson scheme for abstract linear systems}.
	\newblock {\em Numer. Funct. Anal. Optim.}, 28:717--736, 2007.
	
	\bibitem{Casteren2011}
	J.~A. van Casteren.
	\newblock {On the Crank-Nicolson scheme once again}.
	\newblock {\em J. Evol. Equations}, 11:457--476, 2011.
	\newblock Erratum: ibid., 477--483.
	
	\bibitem{Wakaiki2021JEE}
	M.~Wakaiki.
	\newblock {The Cayley transform of the generator of a polynomially stable
		$C_0$-semigroup}.
	\newblock {\em J. Evol. Equations}, 21:4575--4597, 2021.
	
	\bibitem{Wakaiki2022_arXiv}
	M.~Wakaiki.
	\newblock Decay of operator semigroups, infinite-time admissibility, and
	related resolvent estimates.
	\newblock arXiv:2212.00315v1, 2022.
	
	\bibitem{Weidmann1980}
	J.~Weidmann.
	\newblock {\em Linear operators in Hilbert spaces}.
	\newblock New York: Springer, 1980.
	
	\bibitem{Zwart2007SF}
	H.~Zwart.
	\newblock {Growth estimates for $\exp(A^{-1}t)$ on a Hilbert space}.
	\newblock {\em Semigroup Forum}, 74:487--494, 2007.
	
	\bibitem{Zwart2007}
	H.~Zwart.
	\newblock {Is $A^{-1}$ an infinitesimal generator?}
	\newblock In {\em Perspectives in operator theory, Banach Center Publ., vol.
		75, Polish Acad. Sci., Warsaw}, pages 303--313, 2007.
	
\end{thebibliography}
\end{document}